\documentclass[a4paper,10pt]{article}
\usepackage[a4paper,top=15mm,left=25mm,right=25mm,bottom=12mm,includefoot]{geometry}
\usepackage[utf8]{inputenc}
\usepackage[T1]{fontenc}
\usepackage{lmodern,amsmath,amssymb,amsthm,mathtools,dsfont,enumitem,nicefrac,mleftright}
\mleftright

\usepackage[colorlinks,linkcolor=black,citecolor=black,urlcolor=black]{hyperref}
\usepackage[backend=biber]{biblatex}
\AtEveryBibitem{
	\ifentrytype{book}{\clearfield{pages}}{}
}
\DeclareDelimFormat{finalnamedelim}{\ifnumgreater{\value{liststop}}{2}{\finalandcomma}{}\addspace\&\space}
\DeclareDelimFormat{namelabeldelim}{\addnbspace}

\theoremstyle{definition}
\newtheorem{setting/}{Setting}[section]
\newenvironment{setting}{
	\pushQED{\qed}\begin{setting/}}
	{\popQED\end{setting/}
}
\theoremstyle{plain}
\newtheorem{theorem}[setting/]{Theorem}
\newtheorem{lemma}[setting/]{Lemma}
\providecommand{\dd}{\mathop{}\!d}
\providecommand{\defeq}{\mathrel{\vcentcolon=}}
\providecommand{\N}{\ensuremath{\mathds{N}}}
\providecommand{\R}{\ensuremath{\mathds{R}}}
\providecommand{\1}{\mathds{1}}
\providecommand{\E}{\mathbb{E}}
\providecommand{\F}{\ensuremath{\mathbb{F}}}
\providecommand{\Prob}{\ensuremath{\mathbb{P}}}

\providecommand{\MCB}{{\ensuremath{\mathcal{B}}}}
\providecommand{\MCD}{{\ensuremath{\mathcal{D}}}}
\providecommand{\MCE}{{\ensuremath{\mathcal{E}}}}
\providecommand{\MCF}{{\ensuremath{\mathcal{F}}}}
\providecommand{\MCR}{{\ensuremath{\mathcal{R}}}}

\providecommand\given{\nonscript\:\vert\nonscript\:\mathopen{}}
\DeclarePairedDelimiterXPP\Prop[1]{\Prob}(){}{
\renewcommand\given{\nonscript\:\delimsize\vert\nonscript\:\mathopen{}} #1}
\DeclarePairedDelimiterXPP\Exp[1]{\E}[]{}{
\renewcommand\given{\nonscript\:\delimsize\vert\nonscript\:\mathopen{}} #1}
\DeclarePairedDelimiter\abs{\lvert}{\rvert}
\DeclarePairedDelimiter\norm{\lVert}{\rVert}

\makeatletter
\newcommand{\opnorm}{\@ifstar\@opnorms\@opnorm}
\newcommand{\@opnorms}[1]{%
  \left|\mkern-1.5mu\left|\mkern-1.5mu\left|
   #1
  \right|\mkern-1.5mu\right|\mkern-1.5mu\right|
}
\newcommand{\@opnorm}[2][]{%
  \mathopen{#1|\mkern-1.5mu#1|\mkern-1.5mu#1|}
  #2
  \mathclose{#1|\mkern-1.5mu#1|\mkern-1.5mu#1|}
}
\makeatother

\title{Differentiability of semigroups of stochastic differential equations
with Hölder-continuous diffusion coefficients}
\author{Martin Hutzenthaler$^{1}$ \& Daniel Pieper$^{2}$
\bigskip
\\
\small{Faculty of Mathematics, University of Duisburg-Essen, 45117 Essen,
Germany}\\
\small{$^1$e-mail: martin.hutzenthaler@uni-due.de, $^2$e-mail: daniel.pieper@uni-due.de}
}
\date{January 17, 2020}

\begin{document}

\maketitle

\begin{abstract}
	Differentiability of semigroups is useful for many applications.  Here we
	focus on stochastic differential equations whose diffusion coefficient is
	the square root of a differentiable function but not differentiable itself.
	For every $m\in\{0,1,2\}$ we establish an upper bound for a $C^m$-norm of
	the semigroup of such a diffusion in terms of the $C^m$-norms of the drift
	coefficient and of the squared diffusion coefficient.
  The constants in our upper bound are often dimension-independent.
  Our estimates are thus suitable for analyzing certain high-dimensional
	and infinite-dimensional degenerate stochastic differential equations.
\end{abstract}

\let\thefootnote\relax
\footnotetext{\emph{AMS 2010 subject classification:}
60J35 (Primary); 47D07, 60J60 (Secondary)}
\footnotetext{\emph{Key words and phrases:} Markov semigroups, Feller
semigroups, diffusion processes, square-root diffusions, stochastic
differential equations, $C^m$-estimate, generators, degenerate operators,
differentiability}

\section{Introduction}
Let $d\in\N$ and let $X=(X_t)_{t\in[0,\infty)}$ be the
solution of a stochastic differential equation (SDE) with values in
$[0,1]^{d}$.
We prove existence and continuity of spatial derivatives of the functions
$[0,\infty)\times [0,1]^d \ni (t,x)\mapsto (T_t f)(x)\defeq
\Exp{ f(X_t)\given X_0=x} \in \R$, $ f\in C^2([0,1]^{d},\R)$,
under suitable assumptions.
More precisely, Theorem~\ref{thm:semigroup_together} below shows under
suitable assumptions for every $t\in[0,\infty)$ and every $m\in\{0,1,2\}$ that
\begin{equation} \begin{split} \label{eq:intro_main}
		\norm{T_t f}_{C^m}\leq e^{(m^2 \lambda_m+\mu_m)t}\norm{ f}_{C^m},
	\end{split}
\end{equation}
where $\lambda_m$ and $\mu_m$ depend respectively on the partial derivatives of the drift function
and of the squared diffusion function up to order $m$ and where
$\norm{\:\cdot\:}_{C^m}$ is defined in Subsection~\ref{ss:notation} below.
In particular, note that we do not assume differentiability of the diffusion
coefficient but only of the squared diffusion coefficient. The ``cost'' of
allowing square-root diffusions is that we need to assume the diffusion
coefficient matrix to be diagonal; see
Theorem~\ref{thm:semigroup_together} for the
precise setting. We also note that even differentiability of the semigroup is
nontrivial since singular diffusion coefficients (that is, degenerate noise)
can lead to loss of regularity; see Theorem~1.2 in~\textcite{HairerHutzenthalerJentzen2015}.

Partial differentiability of semigroups is used in a number of applications,
e.g.:
\begin{itemize}
	\item inequalities between expectations of diffusions with different
		coefficient functions, e.g.\ Theorem 1 in~\textcite{CoxFleischmannGreven1996}
		or Proposition 2.2 in~\textcite{HutzenthalerWakolbinger2007},
	\item weak convergence rates for numerical approximations of SDEs,
		e.g.\ Theorem 1 in~\textcite{TalayTubaro1990},
	\item stochastic representations of quasilinear parabolic partial
		differential equations, e.g.\ Theorem 3.2 in~\textcite{Peng1991},
\end{itemize}
and many more. These results can now also be derived for those SDEs for which we
establish differentiability of the semigroup.

In the literature, differentiability of semigroups is well-known in the case
of differentiable coefficient functions of suitable order (see, e.g.,
Theorem~8.4.3 in \textcite{GikhmanSkorokhod1969}) and in the case of
one-dimensional SDEs including the case of square-root diffusion coefficients
(see, e.g., \textcite{Dorea1976} or \textcite{Ethier1978}).
Moreover, \textcite{Ethier1976} establishes differentiability of semigroups for a class
of multidimensional SDEs 
with square-root diffusion coefficient
$\{y\in[0,1]^d\colon \sum_{i=1}^d y_i\leq 1\}\ni x
\mapsto (\sqrt{x_i(1-\textstyle\sum_{j=1}^d x_j)})_{i\in\{1,\dotsc,d\}}\in\R^d$.
In addition, Lemma 4.3 in \textcite{EpsteinPop2019} establishes
differentiability of semigroups corresponding to so-called Kimura operators.
So differentiability of semigroups corresponding to degenerate SDEs
is in principle known in the literature.
However, we have not found a result on differentiability of semigroups
corresponding to the specific form of the SDE \eqref{eq:SDE.main-result} beyond the one-dimensional case.

In fact, differentiability of semigroups of degenerate SDEs is not our main concern.
Our main goal is to establish the regularity estimates \eqref{eq:intro_main}
with constants $\lambda_0,\lambda_1,\lambda_2,\mu_0,\mu_1,\mu_2$
that are dimension-independent.
This dimension-independence of regularity estimates of semigroups of degenerate stochastic
differential equations seems to be a new observation.
The benefit of such estimates with dimension-independent constants
is that it allows us to analyze infinite-dimensional (where $d=\infty$)
or high-dimensional (where $d\to \infty$) SDEs.
To mention an example application,
our main result,
Theorem~\ref{thm:semigroup_together} below,
is applied in \textcite{HutzenthalerPieper2018Conv}
to a system of interacting
diffusions on $D\in\N$ demes to obtain that the partial derivatives of the
semigroups are uniformly bounded in $D\in\N$.
This then allows to establish a many-demes limit as $D\to\infty$, that
is, to generalize Theorem 3.3 in~\textcite{Hutzenthaler2012} to a class of
SDEs with nonlinear squared
diffusion coefficients.
In addition, by approximation with finite-dimensional SDEs,
Theorem~\ref{thm:semigroup_together} can
also be applied
to McKean-Vlasov SDEs (e.g.\
(1.2) with $g(x)=x(1-x)$ in~\textcite{DawsonGreven1993}
or
(1.2) in~\textcite{Hutzenthaler2012}
or
(8) in~\textcite{HutzenthalerJordanMetzler2015}).

An important technical insight of this paper is as follows.
Results in the literature are often (e.g., \textcite{Ethier1976} or \textcite{EpsteinPop2019}
with the domain suitably replaced)
formulated in the norms 
\begin{equation}  \begin{split}\label{eq:normvvv}
	C^m([0,1]^d,\R)\ni f\mapsto\opnorm{f}_{C^m([0,1]^d,\R)}
	\defeq\sum_{\alpha \in \N_0^d, \abs{\alpha} \leq m}
	\sup_{x\in[0,1]^d}\abs{\partial^\alpha f(x)}.
\end{split}  \end{equation}
This norm, however, introduces unnecessary dimension-dependence due to the
sum in \eqref{eq:normvvv}. 
To give an illustrative example, if the drift coefficient is $[0,1]^d\ni x\mapsto x\in\R^d$,
if the diffusion coefficient is zero,
and if $f\in C^{1}(\R,\R)$,
then the solution of the SDE \eqref{eq:SDE.main-result} is
$(x_i e^{t})_{t\in[0,\infty), i\in\{1,\ldots,d\}}$ and it holds for all $t\in[0,\infty)$
that
\begin{equation}  \begin{split}
	\opnorm[\bigg]{ [0,1]^d\ni x\mapsto f\Bigl(\sum_{i=1}^d x_ie^t\Bigr)\in\R }_{C^1([0,1]^d,\R)}
	&=\sup_{x\in[0,1]^d}\abs[\Big]{f\Bigl(\sum_{i=1}^d x_ie^t\Bigr)}
	+\sum_{k=1}^d \sup_{x\in[0,1]^d}\abs[\Big]{f'\Bigl(\sum_{i=1}^d
	x_ie^t\Bigr)e^t}
	\\&
	\leq\sup_{z\in\R}\abs{f(z)}+d\sup_{z\in\R}\abs{f'(z)}e^t.
\end{split}  \end{equation}
If the norm $\opnorm{\:\cdot\:}_{C^1([0,1]^d,\R)}$ is replaced
by our norm $\norm{\:\cdot\:}_{C^1}$ where the sum in \eqref{eq:normvvv} is replaced
by the maximum, then the dimension $d$ does not appear on the right-hand side.


\subsection{Notation} \label{ss:notation}
We write $\N_0 \defeq \{0,1,2,\ldots\}$ and $\N \defeq \N_0\setminus\{0\}$.
For every topological space $(E,\MCE)$ we denote by $\MCB(E)$ the Borel
$\sigma$-algebra on $(E,\MCE)$.
For every $d \in \N$ and every $m\in\N_0$ we denote by $C^m([0,1]^d,\R)$ the
set of functions $ f\colon [0,1]^d \to \R$ whose partial derivatives of
order $0$ through $m$ exist and are continuous on $[0,1]^d$.
For every
$d\in\N$ and every $ f\colon [0,1]^d \to \R$ we define $\norm{ f}_\infty
\defeq \sup_{x\in [0,1]^d}\abs{ f(x)} \in [0,\infty]$. For every $d\in \N$
and every multiindex $\alpha = (\alpha_1,\dotsc,\alpha_d) \in \N_0^d$ of
length $\abs{\alpha} \defeq \sum_{k=1}^d \alpha_k$ we write $\partial^\alpha
\defeq \frac{\partial^{\abs \alpha}}{\partial x_1^{\alpha_1}\dotsm\partial
x_d^{\alpha_d}}$. For every $d\in\N$, every $m\in\N_0$, and every $ f \in
C^m([0,1]^d,\R)$ we define $\norm{ f}_{C^m} \defeq \max_{\alpha \in \N_0^d,
\abs{\alpha} \leq m}\norm{\partial^\alpha f}_\infty$. For every $d\in\N$,
every $x = (x_k)_{k\in\{1,\dotsc,d\}} \in [0,1]^d$, and every
$i,j\in\{1,\dotsc,d\}$ we write $\hat x_i \defeq
(x_k)_{k\in\{1,\dotsc,d\}\setminus\{i\}}$ and $\hat x_{ij} \defeq
(x_k)_{k\in\{1,\dotsc,d\}\setminus\{i,j\}}$.

\section{Drift part} \label{sec:drift_term}
In this section, we prove \eqref{eq:intro_main} for $m\in\{0,1,2\}$ and an
analogous result for the $\norm{\:\cdot\:}_{C^3}$-norm under suitable
assumptions in the case where the diffusion coefficient is zero.
The case of non-zero diffusion coefficient is
analyzed in Section \ref{sec:diffusion_term}.
The following
Setting~\ref{set:semigroup_drift} establishes the precise setting assumed
throughout this section.
\begin{setting}[Drift coefficients]\label{set:semigroup_drift}
	Let $d\in\N$, let $b_1,\dotsc,b_d \in C^3([0,1]^d,\R)$ satisfy for all
	$i\in\{1,\dotsc,d\}$ and all $x = (x_1,\dotsc,x_d) \in [0,1]^d$ with $x_i \in
	\{0,1\}$ that $(-1)^{x_i} b_i(x) \geq 0$,
	and for every $m\in\{1,2,3\}$
	we define
	$\lambda_m \defeq \max_{\alpha \in \N_0^d, 0 < \abs{\alpha} \leq
	m}\sum_{i=1}^d \norm{\partial^\alpha b_i}_\infty$.

	Theorem~3.2
	in~\textcite{ShigaShimizu1980} ensures the existence of a deterministic Markov
	process
		$y = (y_1,\dotsc,y_d) \colon [0,\infty) \times [0,1]^d \to [0,1]^d$
	satisfying for all $i\in\{1,\dotsc,d\}$, all $t\in[0,\infty)$, and all $x =
	(x_1,\dotsc,x_d) \in[0,1]^d$ that
	\begin{equation}
		y_i(t,x) = x_i + \int_0^t b_i(y(s,x)) \dd s.
		\label{eq:drift_ODE_integral}
	\end{equation}
	We denote by $\{T^1_t \colon t\in [0,\infty)\}$ the associated
	strongly continuous contraction semigroup on $C([0,1]^d,\R)$, which
	satisfies for all $t\in[0,\infty)$, all $ f \in C([0,1]^d,\R)$, and all
	$x\in [0,1]^d$ that $(T^1_t f)(x) = (f\circ y)(t,x)$.
\end{setting}
\begin{lemma}[$C^1$-esimate for drift part] \label{l:semigroup_drift_C1}
	Assume \textup{Setting~\ref{set:semigroup_drift}} and let $ f \in
	C^1([0,1]^d,\R)$.
	Then it holds for all $t\in[0,\infty)$ that $T^1_t f \in C^1([0,1]^d,\R)$
	and
	\begin{equation}
		\norm{T^1_t f}_{C^1} \leq e^{\lambda_1 t}\norm{ f}_{C^1}.
	\end{equation}
\end{lemma}
\begin{proof}
	The theory of ordinary differential equations yields for all
	$t\in[0,\infty)$ that $y(t,\:\cdot\:) \in C^1([0,1]^d,[0,1]^d)$
	(see, e.g.\ Corollary~V.4.1 in~\textcite{Hartman2002})
	and this together with $ f \in C^1([0,1]^d,\R)$ implies that $T^1_t f
	\in C^1([0,1]^d,\R)$.
	The dominated convergence theorem and~\eqref{eq:drift_ODE_integral} imply
	for all $i, j \in\{1,\dotsc,d\}$, all $t\in[0,\infty)$, and all $x\in
	[0,1]^d$ that
	\begin{equation}
		\frac{\partial y_i}{\partial x_{j}}(t,x) = \1_{i = j} +
		\int_0^t \sum_{k=1}^d \frac{\partial b_i}{\partial
		y_k}(y(s,x))\frac{\partial y_k}{\partial x_{j}}(s,x) \dd s.
	\end{equation}
	It follows for all $j \in\{1,\dotsc,d\}$, all $t\in[0,\infty)$, and all
	$x\in [0,1]^d$ that
	\begin{equation}
		\begin{split}
			\sum_{i=1}^d\abs[\bigg]{\frac{\partial y_i}{\partial x_{j}}(t,x)} 
			&\leq 1 + \int_0^t
			\sum_{k=1}^d
			\left(\sum_{i=1}^d \abs[\bigg]{\frac{\partial b_i}{\partial
			y_k}(y(s,x))}\right)\abs[\bigg]{\frac{\partial y_k}{\partial
			x_{j}}(s,x)} \dd s \\
			&\leq 1 + \int_0^t
			\left(\max_{\alpha \in \N^d_0, \abs{\alpha} = 1}\sum_{i=1}^d
			\norm*{\partial^\alpha
			b_i}_\infty\right)\left(\sum_{k=1}^d\abs[\bigg]{\frac{\partial
			y_k}{\partial x_{j}}(s,x)}\right) \dd s\\
			&= 1 + \int_0^t
			\lambda_1\sum_{k=1}^d\abs[\bigg]{\frac{\partial y_k}{\partial
			x_{j}}(s,x)} \dd s.
		\end{split}
	\end{equation}
	This and Gronwall's inequality yield for all $j
	\in\{1,\dotsc,d\}$, all $t\in[0,\infty)$, and all $x\in [0,1]^d$ that
	\begin{equation}
		\sum_{i=1}^d\abs[\bigg]{\frac{\partial y_i}{\partial x_{j}}(t,x)}
		\leq
		e^{\lambda_1 t}.
		\label{eq:drift_C1_est2}
	\end{equation}
	It follows from the chain rule and from~\eqref{eq:drift_C1_est2} for
	all $j \in\{1,\dotsc,d\}$, all $t\in[0,\infty)$, and all $x\in [0,1]^d$ that
	\begin{equation}
		\begin{split}
			\abs[\bigg]{\frac{\partial (f\circ y)}{\partial x_{j}}(t,x)} =
			\abs[\Bigg]{\sum_{i=1}^d \frac{\partial  f}{\partial y_i}(y(t,x))
			\frac{\partial y_i}{\partial x_{j}}(t,x)}
			\leq \norm{ f}_{C^1}\sum_{i=1}^d\abs[\bigg]{\frac{\partial
			y_i}{\partial x_{j}}
			(t,x)} \leq e^{\lambda_1 t}\norm{ f}_{C^1}.
		\end{split}
	\end{equation}
	Together with the fact that $\sup_{t\in[0,\infty)}\norm{T^1_t f}_\infty
	\leq \norm{ f}_\infty$, this implies for all $t\in[0,\infty)$ that
	\begin{equation}
		\norm{T^1_t f}_{C^1} = \max\left\{\norm{T^1_t f}_\infty,
		\adjustlimits\max_{j
			\in \{1,\dotsc,d\}} \sup_{x\in[0,1]^d} \abs[\bigg]{\frac{\partial
			(f\circ y)}{\partial x_{j}}(t,x)}\right\} \leq e^{\lambda_1 t}\norm{ f}_{C^1}.
	\end{equation}
	This finishes the proof of Lemma~\ref{l:semigroup_drift_C1}.
\end{proof}

\begin{lemma}[$C^2$-estimate for drift part] \label{l:semigroup_drift_C2}
	Assume \textup{Setting~\ref{set:semigroup_drift}} and let $ f \in
	C^2([0,1]^d,\R)$.
	Then it holds for all $t\in[0,\infty)$ that
	$T^1_t f \in C^2([0,1]^d,\R)$ and
	\begin{equation}
		\norm{T^1_t f}_{C^2} \leq e^{4 \lambda_2 t}\norm{ f}_{C^2}.
	\end{equation}
\end{lemma}
\begin{proof}
	The theory of ordinary differential equations yields for all
	$t\in[0,\infty)$ that $y(t,\:\cdot\:) \in C^2([0,1]^d,[0,1]^d)$
	(see, e.g.\ Corollary~V.4.1 in~\textcite{Hartman2002})
	and this together with $ f \in C^2([0,1]^d,\R)$ implies that $T^1_t f
	\in C^2([0,1]^d,\R)$.
	The dominated convergence theorem and~\eqref{eq:drift_ODE_integral} imply
	for all $i, j, k \in \{1,\dotsc,d\}$, all $t\in[0,\infty)$,
	and all $x\in [0,1]^d$ that
	\begin{equation}
		\begin{split}
			\frac{\partial^2 y_i}{\partial x_{k}\partial x_{j}}(t,x) &=
			\int_0^t \sum_{l,m=1}^d \frac{\partial^2 b_i}{\partial y_m \partial
			y_l}(y(s,x))\frac{\partial y_m}{\partial
			x_{k}}(s,x) \frac{\partial y_l}{\partial
			x_{j}}(s,x)
			+ \sum_{l=1}^d \frac{\partial b_i}{\partial y_l}(y(s,x))
			\frac{\partial^2 y_l}{\partial x_{k} \partial x_{j}}(s,x) \dd s.
		\end{split}
	\end{equation}
	This,~\eqref{eq:drift_C1_est2}, and $\lambda_1 \leq \lambda_2$ imply for all $j,
	k \in \{1,\dotsc,d\}$, all $t\in[0,\infty)$, and all $x\in [0,1]^d$
	that
	\begin{equation}
		\begin{split}
			\sum_{i=1}^d \abs[\bigg]{\frac{\partial^2 y_i}{\partial x_{k}\partial x_{j}}
			(t,x)}
			&\leq \int_0^t \sum_{l,m=1}^d \left(\sum_{i=1}^d
			\abs[\bigg]{\frac{\partial^2 b_i}{\partial y_m \partial y_l}(y(s,x))}\right)
			\abs[\bigg]{\frac{\partial y_m}{\partial x_{k}}(s,x)}
			\abs[\bigg]{\frac{\partial y_l}{\partial x_{j}}(s,x)} \\
			&\quad+ \sum_{l=1}^d \left(\sum_{i=1}^d \abs[\bigg]{\frac{\partial
			b_i}{\partial y_l}(y(s,x))}\right)
			\abs[\bigg]{\frac{\partial^2 y_l}{\partial x_{k} \partial x_{j}}(s,x)} \dd s\\
			&\leq \int_0^t \left(\max_{\alpha \in \N_0^d, \abs{\alpha} = 2}\sum_{i=1}^d
			\norm*{\partial^\alpha
			b_i}_\infty\right) \left(\sum_{m=1}^d \abs[\bigg]{\frac{\partial y_m}{\partial
			x_{k}}(s,x)}\right)
			\left(\sum_{l=1}^d \abs[\bigg]{\frac{\partial y_l}{\partial x_{j}}(s,x)}\right) \\
			&\quad+ \left(\max_{\alpha \in \N_0^d,\abs{\alpha} = 1}\sum_{i=1}^d
			\norm*{\partial^\alpha b_i}_\infty\right)
			\left(\sum_{l=1}^d\abs[\bigg]{\frac{\partial^2 y_l}{\partial
			x_{k} \partial x_{j}}(s,x)}\right) \dd s\\
			&\leq \int_0^t \lambda_2 e^{2 \lambda_2 s}
			+ \lambda_2
			\sum_{l=1}^d\abs[\bigg]{\frac{\partial^2 y_l}{\partial x_{k} \partial
			x_{j}}(s,x)} \dd s\\
			& = \tfrac{1}{2}(e^{2\lambda_2t} - 1) + \int_0^t
			\lambda_2\sum_{l=1}^d\abs[\bigg]{\frac{\partial^2 y_l}{\partial x_{k} \partial
			x_{j}}(s,x)} \dd s.
		\end{split}
	\end{equation}
	This and Gronwall's inequality yield for all $j, k \in \{1,\dotsc,d\}$,
	all $t\in[0,\infty)$, and all $x\in [0,1]^d$ that
	\begin{equation}
		\begin{split}
			\sum_{i=1}^d \abs[\bigg]{\frac{\partial^2 y_i}{\partial x_{k}\partial x_{j}}
			(t,x)}
			& \leq \tfrac{1}{2}(e^{2\lambda_2t} - 1)e^{\lambda_2 t}.
		\end{split}
		\label{eq:drift_C2_est2}
	\end{equation}
	It follows from the chain rule, \eqref{eq:drift_C1_est2}, $\lambda_1 \leq \lambda_2$,
	and from~\eqref{eq:drift_C2_est2} for all $j, k
	\in\{1,\dotsc,d\}$, all $t\in[0,\infty)$, and all $x\in [0,1]^d$ that
	\begin{equation}
		\begin{split}
			\abs[\bigg]{\frac{\partial^2 ( f \circ y)}{\partial x_{k}\partial x_{j}}(t,x)} &\leq
			\abs[\Bigg]{\sum_{i,l=1}^d \frac{\partial^2  f}{\partial y_l \partial
			y_i}(y(t,x))\frac{\partial y_l}{\partial x_{k}}(t,x)\frac{\partial
			y_i}{\partial x_{j}}(t,x)}
			+\abs[\Bigg]{\sum_{i=1}^d \frac{\partial  f}{\partial
			y_i}(y(t,x))\frac{\partial^2 y_i}{\partial
			x_{k} \partial x_{j}}(t,x)}\\
			&\leq \norm{ f}_{C^2}\Biggl( \sum_{l=1}^d
			\abs[\bigg]{\frac{\partial y_l}{\partial x_{k}}(t,x)}
			\Biggr)\Biggl( \sum_{i=1}^d \abs[\bigg]{\frac{\partial y_i}{\partial
			x_{j}}(t,x)} \Biggr)
			+ \norm{ f}_{C^2} \sum_{i=1}^d
			\abs[\bigg]{\frac{\partial^2 y_i}{\partial x_{k} \partial x_{j}}(t,x)}\\
			&\leq \left(e^{2 \lambda_2 t} + \tfrac{1}{2}(e^{2 \lambda_2 t}-1)e^{\lambda_2 t}\right)
			\norm{ f}_{C^2} \\
			&\leq \left(e^{2 \lambda_2 t} + (e^{2 \lambda_2 t} - 1)e^{2\lambda_2
			t}\right)\norm{ f}_{C^2} = e^{4 \lambda_2 t}\norm{ f}_{C^2}.
		\end{split}
	\end{equation}
	Together with Lemma~\ref{l:semigroup_drift_C1} and $\lambda_1\leq \lambda_2$, this shows for all
	$t\in[0,\infty)$ that
	\begin{equation}
		\norm{T^1_t f}_{C^2} = \max\left\{\norm{T^1_t f}_{C^1},
		\adjustlimits\max_{j, k
		\in\{1,\dotsc,d\}}\sup_{x\in[0,1]^d}\abs[\bigg]{\frac{\partial^2( f
		\circ y)}{\partial
		x_{k}\partial
		x_{j}}(t,x)}\right\}\leq e^{4\lambda_2 t}
		\norm{ f}_{C^2}.
	\end{equation}
	This finishes the proof of Lemma~\ref{l:semigroup_drift_C2}.
\end{proof}

The proof of the following Lemma~\ref{l:semigroup_drift_C3} is analogous to
the proof of Lemma~\ref{l:semigroup_drift_C2} and therefore omitted here.
\begin{lemma}[$C^3$-estimate for drift part] \label{l:semigroup_drift_C3}
	Assume \textup{Setting~\ref{set:semigroup_drift}} and let $ f \in
	C^3([0,1]^d,\R)$.
	Then it holds for all $t\in[0,\infty)$ that
	$T^1_t f \in C^3([0,1]^d,\R)$ and
	\begin{equation}
		\norm{T^1_t f}_{C^3} \leq e^{13 \lambda_3 t}\norm{ f}_{C^3}.
	\end{equation}
\end{lemma}

\section{Diffusion part} \label{sec:diffusion_term}
The goal of this section is to prove \eqref{eq:intro_main} for
$m\in\{0,1,2,3\}$ under suitable assumptions in the case where the drift
coefficient is zero; see Lemma~\ref{l:semigroup_diffusion} below. For that, we
first look at the one-dimensional case in Subsection~\ref{ss:diffusion_1_dim}
below, and then we lift this result to the multidimensional case in
Subsection~\ref{ss:diffusion_d_dim} below.

\subsection{One-dimensional case} \label{ss:diffusion_1_dim}
The following lemma on smoothness preservation of the semigroup is well-known
if, for $m\in \{0,1,2,3\}$, the norm $\norm{\:\cdot\:}_{C^m}$ is replaced by the
equivalent norm $\varphi \mapsto \sum_{k=0}^m \norm{\frac{d^k
\varphi}{dx^k}}_\infty$; see~\textcite{Dorea1976}. The proof of the new upper
bound of the operator norm of the semigroup with respect to
$\norm{\:\cdot\:}_{C^m}$ for $m\in\{0,1,2,3\}$ is a straightforward adaptation of
the proofs in~\textcite{Dorea1976}.
\begin{lemma}[Smoothness preservation of one-dimensional diffusive part] \label{l:semigroup_1d}
	Let $a\in C^3([0,1],\R)$ satisfy that $a(0) = 0 = a(1)$ and for all
	$x\in(0,1)$ that $a(x) > 0$, let $A \colon C^2([0,1],\R) \to C([0,1],\R)$
	satisfy for all $\varphi \in C^2([0,1],\R)$ and all $x\in [0,1]$ that
	\begin{equation}
		(A\varphi)(x) = \frac{1}{2} a(x)\frac{d^2 \varphi}{dx^2}(x), 
	\end{equation}
	for all $m\in\N_0$ we define $\MCD_m(A) \defeq C^2([0,1],\R) \cap C^m([0,1],\R)
	\cap A^{-1}C^m([0,1],\R)$, 
	we define $\nu_0 \defeq 0$, $\nu_1 \defeq 0$, $\nu_2 \defeq \frac{1}{2} \norm{\frac{d^2
	a}{dx^2}}_\infty$, and $\nu_3 \defeq \norm{\frac{d^3
	a}{dx^3}}_\infty + \frac{3}{2}\norm{\frac{d^2 a}{dx^2}}_\infty$,
	and we denote by $\{S_t\colon t\in[0,\infty)\}$ the
	strongly continuous contraction semigroup on $C([0,1],\R)$ generated by
	$(A,\MCD_0(A))$; see \textup{Theorem~1} on \textup{p.~38}
	in~\textup{\textcite{Mandl1968}}.
	Then it holds for all
	$m\in\{0,1,2,3\}$ that
	\begin{enumerate}[label=\textup{(\roman*)}]
		\item it holds for all $t\in[0,\infty)$ that $S_t\colon C^m([0,1],\R) \to
			C^m([0,1],\R)$,
		\item $\{S_t \colon t\in[0,\infty)\}$ defines a strongly
			continuous semigroup on $C^m([0,1],\R)$ with generator $(A,\MCD_m(A))$,
			and
		\item it holds for all $t \in [0,\infty)$ and all $\varphi\in C^m([0,1],\R)$ that
			\begin{equation}
				\norm*{S_t\varphi}_{C^m} \leq e^{\nu_m t}\norm{\varphi}_{C^m}.
				\label{eq:estimate_semigroup}
			\end{equation}
	\end{enumerate}
\end{lemma}
\begin{proof}
	For every $m\in\{0,1,2,3\}$ Theorem~1 and Remark~1 in~\textcite{Ethier1978}
	and the Main Theorem in~\textcite{Dorea1976} yield for all $t\in[0,\infty)$ that $S_t \colon
	C^m([0,1],\R) \to C^m([0,1],\R)$ and that $\{S_s \colon s\in[0,\infty)\}$
	restricted to $C^m([0,1],\R)$ defines a strongly continuous semigroup with
	generator $(A,\MCD_m(A))$. This proves (i) and (ii).

	It remains to check that~\eqref{eq:estimate_semigroup} can be established
	with our choice of the norm on $C^m([0,1],\R)$.
	For every $m \in \{0,1,2\}$ Theorem~$m$ in~\textcite{Dorea1976} yields for
	all $\lambda > \nu_m$ and all $\varphi \in C^m([0,1],\R)$ that $J_\lambda\varphi \defeq
	(\lambda - A)^{-1}\varphi \in \MCD_m(A)$ exists and its proof shows that
	\begin{equation}
		\norm[\big]{\tfrac{d^m J_\lambda \varphi}{dx^m}}_{\infty} \leq
		\tfrac{1}{\lambda - \nu_m}\norm[\big]{\tfrac{d^m \varphi}{dx^m}}_{\infty}.
		\label{eq:resolvent_estimate}
	\end{equation}
	Fix $m\in\{0,1,2\}$ for the rest of this paragraph. Consider $G \defeq A - \nu_m$
	with domain $\MCD(G) = \MCD_m(A)$. Since $C^\infty([0,1],\R) \subseteq \MCD(G)$, it
	follows that $\MCD(G)$ is dense in $C^m([0,1],\R)$ w.r.t.~$\norm{\:\cdot\:}_{C^m}$.
	Equation~\eqref{eq:resolvent_estimate} implies for all
	$\lambda,\lambda^\prime > 0$ with $\lambda = \lambda^\prime + \nu_m$ and all
	$\varphi \in C^m([0,1],\R)$ that $(\lambda^\prime - G)^{-1}\varphi = J_\lambda\varphi
	\in \MCD(G)$ and
	\begin{equation}
		\begin{split}
			\norm{(\lambda^\prime - G)^{-1}\varphi}_{C^m} =
			\norm{J_\lambda\varphi}_{C^m} &= \max_{k\in\{0,\dotsc,m\}}
			\norm[\big]{\tfrac{d^kJ_\lambda\varphi}{dx^k}}_{\infty} \\
			&\leq \max_{k\in\{0,\dotsc,m\}}
			\tfrac{1}{\lambda-\nu_k}\norm[\big]{\tfrac{d^k \varphi}{dx^k}}_{\infty} \leq
			\tfrac{1}{\lambda-\nu_m}\norm*{\varphi}_{C^m} =
			\tfrac{1}{\lambda^\prime}\norm*{\varphi}_{C^m}.
		\end{split}
	\end{equation}
	Thus $\MCD(G)$ is dense in $C^m([0,1],\R)$, $G$ is dissipative, and
	$\MCR(1-G) = C^m([0,1],\R)$. Consequently, the Hille-Yosida theorem (see, e.g.\
	Theorem~1.2.6 in~\textcite{EthierKurtz1986}) yields that $G$ generates a unique strongly continuous
	contraction semigroup $\{P_t \colon t \in [0,\infty)\}$ on $C^m([0,1],\R)$. This
	implies that $\{e^{\nu_m t}P_t \colon t\in[0,\infty)\}$ is a strongly continuous semigroup on
	$C^m([0,1],\R)$ with infinitesimal generator $\nu_m + G = A$. It follows that
	$\{S_t \colon t\in[0,\infty)\}$ restricted to $C^m([0,1],\R)$ is given by
	$\{e^{\nu_m t}P_t \colon t\in[0,\infty)\}$ and that it holds for
	all $t\in[0,\infty)$ and all $\varphi \in C^m([0,1],\R)$ that
	\begin{equation}
		\norm{S_t\varphi}_{C^m} = e^{\nu_m t}\norm{P_t\varphi}_{C^m} \leq
		e^{\nu_m t}\norm{\varphi}_{C^m}.
	\end{equation}
	Since $m\in\{0,1,2\}$ was arbitrary, \eqref{eq:estimate_semigroup} is shown for all $m
	\in\{0,1,2\}$.

	To prove (iii), it remains to treat the case $m=3$. Define $\tilde \nu_3 \defeq \nu_3 -
	\frac{1}{2}\norm{\frac{d^3 a}{dx^3}}_\infty$. Theorem~3 in~\textcite{Dorea1976}
	yields for all $\lambda > \tilde \nu_3$ and all
	$\varphi \in C^3([0,1],\R)$ that $J_\lambda\varphi \defeq (\lambda - A)^{-1}\varphi \in
	\MCD_3(A)$ exists and its proof shows that
	\begin{equation}
		\norm[\big]{\tfrac{d^3J_\lambda\varphi}{dx^3}}_{\infty} \leq
		\tfrac{1}{\lambda - \tilde \nu_3}\Bigl(\norm[\big]{\tfrac{d^3 \varphi}{dx^3}}_{\infty}
		+
		\tfrac{1}{2}\norm[\big]{\tfrac{d^3 a}{dx^3}}_{\infty}
		\norm[\big]{\tfrac{d^2J_\lambda\varphi}{dx^2}}_\infty \Bigr).
	\end{equation}
	This,~\eqref{eq:resolvent_estimate}, and the inequality $\nu_0\leq \nu_1 \leq
	\nu_2 \leq \tilde \nu_3$ yield for all $\lambda > \tilde \nu_3$ and
	all $\varphi\in C^3([0,1],\R)$ that
	\begin{equation}
		\norm*{J_\lambda\varphi}_{C^3} \leq
		\tfrac{1}{\lambda - \tilde \nu_3}\Bigl(\norm*{\varphi}_{C^3}
		+ \tfrac{1}{2}\norm[\big]{\tfrac{d^3 a}{dx^3}}_{\infty}
		\norm{J_\lambda\varphi}_{C^3} \Bigr).
		\label{eq:resolvente}
	\end{equation}
	If $\lambda > \nu_3$, then $\lambda > \tilde \nu_3$ and $1 -
	\frac{1}{2}\norm{\frac{d^3 a}{dx^3}}_\infty(\lambda -
	\tilde \nu_3)^{-1} = \frac{\lambda - \nu_3}{\lambda - \tilde \nu_3} > 0$,
	rearranging~\eqref{eq:resolvente} therefore yields for all $\lambda >
	\nu_3$ and all $\varphi \in C^3([0,1],\R)$ that
	\begin{equation}
		\norm*{J_\lambda\varphi}_{C^3}
		\leq \tfrac{\lambda - \tilde \nu_3}{\lambda - \nu_3}\tfrac{1}{\lambda - \tilde
		\nu_3}\norm*{\varphi}_{C^3}
		= \tfrac{1}{\lambda - \nu_3}\norm*{\varphi}_{C^3}.
	\end{equation}
	The remaining part of the proof of (iii) follows from an application of the
	Hille-Yosida theorem as in the previous paragraph. This finishes the proof
	of Lemma~\ref{l:semigroup_1d}.
\end{proof}
\subsection{Multidimensional case}
\label{ss:diffusion_d_dim}
Throughout this subsection, we use the definitions
and the notation introduced in the following
Setting~\ref{set:diffusion_ddim}.
\begin{setting}[Diffusion coefficients]\label{set:diffusion_ddim}
	Let $d\in\N$, let $(\Omega,\MCF,\Prob, (\F_t)_{t\in[0,\infty)})$ be a
	stochastic basis, let $W = (W(1),\dotsc,W(d)) \colon [0,\infty) \times \Omega
	\to \R^d$ be a standard $(\F_t)_{t\in[0,\infty)}$-Brownian motion with
	continuous sample paths, let $a_1,\dotsc,a_d \in C^3([0,1],\R)$ satisfy for
	all $i\in\{1,\dotsc,d\}$ and all $x\in(0,1)$ that $a_i(0) = 0 = a_i(1)$
	and $a_i(x) > 0$,
	and we define
	$\mu_0 \defeq 0$, $\mu_1 \defeq 0$, 
	$\mu_2 \defeq
	\max_{i\in\{1,\dotsc,d\}}\frac{1}{2}\norm{\frac{d^2 a_i}{dx^2}}_\infty$, and
	$\mu_3 \defeq \max_{i\in\{1,\dotsc,d\}}(\norm{\frac{d^3 a_i}{dx^3}}_\infty +
	\frac{3}{2}\norm{\frac{d^2 a_i}{dx^2}}_\infty)$.

	Theorem~3.2 in \textcite{ShigaShimizu1980} implies that there exist
	$(\F_t)_{t\in[0,\infty)}$-adapted processes
	$Y^x = (Y^x(1),\dotsc,Y^x(d)) \colon [0,\infty)\times\Omega \to [0,1]^d$,
	$x\in[0,1]^d$,
	with continuous sample paths satisfying for all $i\in\{1,\dotsc,d\}$,
	all $t\in[0,\infty)$, and all $x = (x_1,\dotsc,x_d) \in [0,1]^d$ that $\Prob$-a.s.
	\begin{equation}
		Y^x_t(i) = x_i + \int_0^t \sqrt{a_i(Y^x_s(i))} \dd W_s(i).
	\end{equation}
	We denote by $\{T^2_t \colon t\in[0,\infty)\}$ the associated strongly continuous contraction
	semigroup on $C([0,1]^d,\R)$, which satisfies for all $t\in[0,\infty)$, all
	$f\in C([0,1]^d,\R)$, and all $x\in[0,1]^d$ that $(T^2_t f)(x) =
	\Exp{f(Y^x_t)}$; see Remark~3.2 in
	\textcite{ShigaShimizu1980}.
	For every $i\in\{1,\dotsc,d\}$ we denote by $\{S^i_t \colon t\in[0,\infty)\}$ the
	strongly continuous contraction semigroup on $C([0,1],\R)$ associated with
	$Y^{\cdot}(i)$, which satisfies for all $t\in[0,\infty)$, all $\varphi \in
	C([0,1],\R)$, and all $x \in [0,1]$ that $(S^i_t \varphi)(x) =
	\Exp{\varphi(Y^x_t(i))}$, and by
	\begin{equation}
		[0,\infty) \times [0,1] \times \MCB(\R) \ni (t,x,A) \mapsto
		p^i_t(x,A) \in [0,1]
	\end{equation}
	the corresponding transition kernel.

	Note that $Y^{\cdot}(i)$, $i\in\{1,\dotsc,d\}$, are
	independent diffusion processes with generators $A_i \colon
	C^2([0,1],\R) \to C([0,1],\R)$, $i\in\{1,\dotsc,d\}$, satisfying for all
	$i\in\{1,\dotsc,d\}$, all $\varphi \in C^2([0,1],\R)$, and all $x\in[0,1]$ that
	\begin{equation}
		(A_i\varphi)(x) = \frac{1}{2} a_i(x)\frac{d^2 \varphi}{dx^2}(x),
	\end{equation}
	so that the result of Subsection~\ref{ss:diffusion_1_dim} applies.
	Moreover, it holds for all $i\in\{1,\dotsc,d\}$, all $t\in[0,\infty)$, all
	$\varphi \in C([0,1],\R)$, and all $x\in[0,1]$ that
	\begin{equation}
		(S^i_t\varphi)(x)
		= \int p^i_t(x,dy)\varphi(y)
	\end{equation}
	and it holds for all $t\in[0,\infty)$, all
	$ f \in C([0,1]^d,\R)$, and all $x = (x_1,\dotsc,x_d) \in[0,1]^d$ that
	\begin{equation}
		(T^2_t f)(x) = \int \bigotimes_{k=1}^d p^k_t(x_k,dy_k) f(y).\qedhere
	\end{equation}
\end{setting}

The aim of this subsection is to show for all $m\in\{0,1,2,3\}$ that it
holds for all $t\in[0,\infty)$ that $T^2_t\colon C^m([0,1]^d,\R) \to
C^m([0,1]^d,\R)$ and for all $t\in[0,\infty)$ and all $ f\in
C^m([0,1]^d,\R)$ that $\norm{T^2_t f}_{C^m} \leq e^{\mu_m
t}\norm{ f}_{C^m}$; see Lemma~\ref{l:semigroup_diffusion} below.
\begin{lemma}[Continuity property] \label{l:continuity_semigroup1}
	Assume \textup{Setting~\ref{set:diffusion_ddim}}, let
	$t\in[0,\infty)$, let $f \in C([0,1]^{d},\R)$, and let
	$I\subseteq\{1,\dotsc,d\}$.
	Then the function
	\begin{equation}
		[0,1]^{d} \ni x \mapsto \int
		\bigotimes_{k\in\{1,\dotsc,d\}\setminus I} p^k_t(x_k, dy_k)
		 f\bigl( (x_i \1_{i \in I} + y_i \1_{i \not\in I})_{i\in\{1,\dotsc,d\}}
		\bigr)
	\end{equation}
	is continuous.
\end{lemma}
\begin{proof}
	Throughout this proof, we denote by $ f_I \colon [0,1]^d \times [0,1]^d \to
	\R$ the function satisfying for all $x,y\in[0,1]^d$ that $ f_I(x,y) =
	 f( (x_i \1_{i \in I} + y_i \1_{i \not\in I})_{i\in\{1,\dotsc,d\}} )$.
	Let $\{x^n \colon n\in\N\} \subseteq [0,1]^d$ be a convergent
	sequence with $\lim_{n\to\infty}x^n = x \in [0,1]^{d}$. Then it holds for all $n\in\N$ that
	\begin{equation}
		\begin{split}
			\MoveEqLeft \abs*{\int \bigotimes_{k\in\{1,\dotsc,d\}\setminus I}
			p^k_t(x_k^n,dy_k) f_I(x^n,y) - \int
			\bigotimes_{k\in\{1,\dotsc,d\}\setminus I} p^k_t(x_k,dy_k) f_I(x,y)}\\
			&\leq \abs*{\int \bigotimes_{k\in\{1,\dotsc,d\}\setminus I}
			p^k_t(x_k^n,dy_k)\bigl( f_I(x^n,y) -  f_I(x,y)\bigr)}\\
			&\quad+ \abs*{\int \bigotimes_{k\in\{1,\dotsc,d\}\setminus I}
			p^k_t(x_k^n,dy_k) f_I(x,y) - \int
			\bigotimes_{k\in\{1,\dotsc,d\}\setminus I} p^k_t(x_k,dy_k) f_I(x,y)}\\
			&\leq \sup_{y\in [0,1]^{d}}\abs[\big]{ f_I(x^n,y) -  f_I(x,y)} \\
			&\quad+ \abs*{\int \bigotimes_{k\in\{1,\dotsc,d\}\setminus I}
			p^k_t(x_k^n,dy_k) f_I(x,y) - \int
			\bigotimes_{k\in\{1,\dotsc,d\}\setminus I} p^k_t(x_k,dy_k) f_I(x,y)}.
		\end{split}
	\end{equation}
	By uniform continuity of $ f$ on $[0,1]^{d}$,
	the first summand on the right-hand side converges to zero as $n\to\infty$.
	For fixed $x \in [0,1]^d$, the function $[0,1]^{d} \ni y \mapsto
	 f_I(x,y)$ is
	continuous, which implies the continuity of
	$[0,1]^{d} \ni z \mapsto \int \bigotimes_{k\in\{1,\dotsc,d\}\setminus I}
	p^k_t(z_k,dy_k) f_I(x,y)$.  Therefore, the second summand on the
	right-hand side converges to
	zero as $n\to\infty$. This finishes the proof of
	Lemma~\ref{l:continuity_semigroup1}.
\end{proof}
\begin{lemma}[Continuity of pure derivatives] \label{l:continuity_semigroup2}
	Assume \textup{Setting~\ref{set:diffusion_ddim}}, let
	$m\in\{0,1,2,3\}$, let $t\in[0,\infty)$, and let $f \in C^m([0,1]^d,\R)$. 
	Then it holds for every
	$i\in\{1,\dotsc,d\}$ that the partial derivative
	\begin{equation}
		[0,1]^d \ni x \mapsto \frac{\partial^m}{\partial x_i^m} \int
		p^i_t(x_i,dy_i) f(x_1,\dotsc,x_{i-1},y_i,x_{i+1},\dotsc,x_d)
		\label{eq:continuity_semigroup2}
	\end{equation}
	exists and is continuous.
\end{lemma}
\begin{proof}
	It suffices to prove the claim for $i=1$.
	For fixed $x \in [0,1]^{d}$, the function $[0,1] \ni y \mapsto  f(y,\hat
	x_1)$ is in $C^m([0,1],\R)$, so Lemma~\ref{l:semigroup_1d} implies that the function
	$[0,1] \ni z \mapsto \int p^1_t(z,dy_1) f(y_1,\hat x_1)$ is in $C^m([0,1],\R)$. This
	shows the existence of the partial derivative~\eqref{eq:continuity_semigroup2}.
	It remains to show continuity on $[0,1]^d$. For that,
	let $\{x^n\colon n\in\N\}\subseteq [0,1]^d$
	be a convergent sequence with $\lim_{n\to\infty} x^n = x \in [0,1]^d$. 
	Lemma~\ref{l:semigroup_1d} implies for all $n\in\N$ that
	\begin{equation}
		\begin{split}
			\MoveEqLeft\abs*{\frac{\partial^m}{\partial (x^n_1)^m}\int p^1_t(x_1^n,dy_1)
			\left( f(y_1,\widehat{x^n}_1) -
			 f(y_1,\hat x_1)\right)} \\
			&\leq e^{\mu_m t}
			\adjustlimits \max_{k\in\{0,\dotsc,m\}} \sup_{z \in
			[0,1]}\,\abs*{\frac{\partial^k  f}{\partial
			z^k}(z,\widehat{x^n}_1) - \frac{\partial^k  f}{\partial
			z^k}(z, \hat x_1)}.
		\end{split}
		\label{eq:semigroup_partialder}
	\end{equation}
	Since $ f \in C^m([0,1]^d,\R)$, it follows for all $k\in\{0,\dotsc,m\}$ that
	$[0,1]^d \ni x \mapsto \frac{\partial^k  f}{\partial x_1^k}(x)$ is uniformly
	continuous.
	Therefore, the right-hand side of~\eqref{eq:semigroup_partialder} converges to
	zero as $n\to\infty$. 
	It holds for all $n\in\N$ that
	\begin{equation}
		\begin{split}
			\MoveEqLeft[3] \abs*{\frac{\partial^m}{\partial (x^n_1)^m}\int p^1_t(x_1^n,dy_1)
			 f(y_1,\widehat{x^n}_1) - \frac{\partial^m}{\partial x_1^m}\int
			p^1_t(x_1,dy_1)  f(y_1,\hat x_1)}\\
			\leq {} &
			\abs*{\frac{\partial^m}{\partial (x^n_1)^m}\int p^1_t(x_1^n,dy_1)
			\bigl( f(y_1,\widehat{x^n}_1) -  f(y_1,\hat x_1)\bigr)} \\
			&+ \abs*{\frac{\partial^m}{\partial (x^n_1)^m}\int p^1_t(x_1^n,dy_1)
			 f(y_1,\hat x_1) - \frac{\partial^m}{\partial x_1^m}\int
			p^1_t(x_1,dy_1)  f(y_1,\hat x_1)}
		\end{split}
		\label{eq:semigroup_partialder2}
	\end{equation}
	The first summand on the right-hand side of~\eqref{eq:semigroup_partialder2}
	converges to zero as $n\to\infty$ by~\eqref{eq:semigroup_partialder}.
	We have shown above that $[0,1] \ni z \mapsto \int p^1_t(z,dy_1) f(y_1,\hat x_1)$ is in
	$C^m([0,1],\R)$, so also the second summand on the right-hand side
	of~\eqref{eq:semigroup_partialder2} converges to zero as $n\to\infty$. This
	finishes the proof of Lemma~\ref{l:continuity_semigroup2}.
\end{proof}
\begin{lemma}[Continuity of pure derivatives, continued] \label{l:semigroup_partial_diag}
	Assume \textup{Setting~\ref{set:diffusion_ddim}}, let
	$m\in\{0,1,2,3\}$, let $t\in[0,\infty)$, and let $f \in C^m([0,1]^d,\R)$.
	Then it holds for every $i\in\{1,\dotsc,d\}$ that the partial derivative
	\begin{equation}
		[0,1]^d \ni x \mapsto \frac{\partial^m}{\partial x_i^m} \int \bigotimes_{k=1}^d
		p^k_t(x_k,dy_k) f(y)
		\label{eq:semigroup_partial_diag}
	\end{equation}
	exists and is continuous.
\end{lemma}
\begin{proof}
	It suffices to show the claim for $i=1$.
	By Fubini's theorem, it holds for all $x\in [0,1]^d$ that
	\begin{equation}
		\int \bigotimes_{k=1}^d p^k_t(x_k,dy_k) f(y)
		= \int p^1_t(x_1,dy_1) \int \bigotimes_{k=2}^d p^k_t(x_k,dy_k)
		 f(y).
		\label{eq:semigroup_existence_diag}
	\end{equation}
	For fixed $x \in [0,1]^{d}$, the fact that $ f \in C^m([0,1]^d,\R)$ and
	the dominated convergence theorem imply that the function $[0,1] \ni z\mapsto
	\int \bigotimes_{k=2}^d p^k_t(x_k,dy_k) f(z,\hat y_1)$ is in $C^m([0,1],\R)$.
	Therefore,~\eqref{eq:semigroup_existence_diag} and
	Lemma~\ref{l:semigroup_1d} prove
	the existence of the partial derivative~\eqref{eq:semigroup_partial_diag}.
	Moreover, Fubini's theorem, the fact that $ f\in C^m([0,1]^d,\R)$,
	Lemma~\ref{l:semigroup_1d}, and the
	dominated convergence theorem imply for all $x\in [0,1]^d$ that
	\begin{equation}
		\begin{split}
			\frac{\partial^m}{\partial x_1^m}\int \bigotimes_{k=1}^d
			p^k_t(x_k,dy_k) f(y)
			&= \frac{\partial^m}{\partial x_1^m}\int \bigotimes_{k=2}^d p^k_t(x_k,
			dy_k)\int p^1_t(x_1,dy_1) f(y)\\
			&= \int \bigotimes_{k=2}^d p^k_t(x_k,dy_k)
			\frac{\partial^m}{\partial x_1^m}\int p^1_t(x_1,dy_1) f(y).
		\end{split}
	\end{equation}
	Consequently, Lemma~\ref{l:continuity_semigroup2} and
	Lemma~\ref{l:continuity_semigroup1} imply the continuity
	of~\eqref{eq:semigroup_partial_diag}.
	This completes the proof of
	Lemma~\ref{l:semigroup_partial_diag}.
\end{proof}
\begin{lemma}[Continuity of mixed second derivatives] \label{l:semigroup_mixed_2nd}
	Assume \textup{Setting~\ref{set:diffusion_ddim}} and let
	$t\in[0,\infty)$ and $f \in C^2([0,1]^d,\R)$. Then it holds for every
	$i,j\in\{1,\dotsc,d\}$ that the partial derivative
	\begin{equation}
		[0,1]^d \ni x \mapsto \frac{\partial^2}{\partial x_i\partial x_j} \int
		\bigotimes_{k=1}^d
		p^k_t(x_k, dy_k) f(y)
		\label{eq:semigroup_mixed_2nd}
	\end{equation}
	exists and is continuous.
\end{lemma}
\begin{proof}
	The case where $i=j$ is treated by Lemma~\ref{l:semigroup_partial_diag}. It
	suffices to consider $i=1$ and $j=2$.
	The dominated convergence theorem
	implies for all $x\in [0,1]^d$ that
	\begin{equation}
		\frac{\partial}{\partial x_1} \int \bigotimes_{k=2}^d
		p^k_t(x_k,dy_k) f(x_1,\hat y_1)
		= \int \bigotimes_{k=2}^d
		p^k_t(x_k,dy_k)\frac{\partial  f}{\partial x_1} (x_1,\hat
		y_1).
		\label{eq:semigroup_mixed1}
	\end{equation}
	Using~\eqref{eq:semigroup_mixed1} and Fubini's theorem, it follows for all
	$x\in [0,1]^d$ that
	\begin{equation}
		\frac{\partial}{\partial x_1} \int
		\bigotimes_{k=2}^d p^k_t(x_k,dy_k) f(x_1, \hat y_1)
		= \int p^2_t(x_2,dy_2) \int
		\bigotimes_{k=3}^d p^k_t(x_k,dy_k)\frac{\partial  f}{\partial x_1}
		(x_1, \hat y_1).
		\label{eq:semigroup_mixed2}
	\end{equation}
	For fixed $x \in [0,1]^{d}$, the fact that $ f \in C^2([0,1]^d,\R)$ and
	the dominated convergence theorem imply that the function $[0,1] \ni z \mapsto
	\int \bigotimes_{k=3}^d p^k_t(x_k,dy_k) \frac{\partial  f}{\partial
	x_1}(x_1,z,\hat y_{12})$ is in $C^1([0,1],\R)$.
	Therefore,~\eqref{eq:semigroup_mixed2} and
	Lemma~\ref{l:semigroup_1d} imply the existence of the partial derivative
	$[0,1]^d \ni x \mapsto \frac{\partial^2}{\partial x_2 \partial x_1}\int
	\bigotimes_{k=2}^d p^k_t(x_k,dy_k) f(x_1,\hat y_1)$.
	Fubini's theorem, Lemma~\ref{l:semigroup_1d}, and the dominated convergence
	theorem imply for all $x\in [0,1]^d$ that
	\begin{equation}
		\frac{\partial^2}{\partial x_2 \partial x_1} \int
		\bigotimes_{k=2}^d p^k_t(x_k,dy_k) f(x_1,\hat y_1)
		= \int \bigotimes_{k=3}^d
		p^k_t(x_k,dy_k) \frac{\partial}{\partial x_2} \int p^2_t(x_2,dy_2)
		\frac{\partial  f}{\partial x_1}(x_1,\hat
		y_1).
		\label{eq:semigroup_mixed3}
	\end{equation}
	Lemma~\ref{l:continuity_semigroup2} and Lemma~\ref{l:continuity_semigroup1}
	show that~\eqref{eq:semigroup_mixed3} is continuous as a function of $x
	\in [0,1]^d$.
	Consequently, Schwarz's theorem (see, e.g.\ Theorem~9.41 in~\textcite{Rudin1976})
	implies that the partial derivative
	$[0,1]^d \ni x \mapsto \frac{\partial^2}{\partial x_1 \partial x_2}\int
	\bigotimes_{k=2}^d p^k_t(x_k,dy_k) f(x_1,\hat y_1)$ exists and satisfies
	for all $x\in[0,1]^d$ that
	\begin{equation}
		\frac{\partial^2}{\partial x_1 \partial x_2} \int
		\bigotimes_{k=2}^d p^k_t(x_k,dy_k) f(x_1,\hat y_1)
		= \frac{\partial^2}{\partial x_2 \partial x_1} \int
		\bigotimes_{k=2}^d p^k_t(x_k,dy_k) f(x_1,\hat y_1).
	\end{equation}
	In particular, for fixed $x \in [0,1]^{d}$, the
	function $z \mapsto \frac{\partial}{\partial x_2}\int \bigotimes_{k=2}^d
	p^k_t(x_k, dy_k) f(z, \hat y_1)$ is in $C^1([0,1],\R)$. From this and
	Lemma~\ref{l:semigroup_1d}, it follows that the partial
	derivative~\eqref{eq:semigroup_mixed_2nd}
	exists.
	Fubini's theorem, Lemma~\ref{l:semigroup_1d}, and the dominated convergence
	theorem further show for all $x\in [0,1]^d$ that
	\begin{equation}
		\frac{\partial^2}{\partial x_1 \partial x_2} \int
		\bigotimes_{k=1}^d p^k_t(x_k,dy_k)  f(y)
		= \int \bigotimes_{k=3}^d p^k_t(x_k,dy_k) \frac{\partial}{\partial x_1}\int p^1_t(x_1,dy_1)
		\frac{\partial}{\partial x_2}\int p^2_t(x_2,dy_2) f(y).
		\label{eq:semigroup_second_der}
	\end{equation}
	Then Lemma~\ref{l:continuity_semigroup2} and
	Lemma~\ref{l:continuity_semigroup1} imply
	that~\eqref{eq:semigroup_second_der} is continuous as a function of $x \in
	[0,1]^d$. This concludes the proof of Lemma~\ref{l:semigroup_mixed_2nd}.
\end{proof}
The proof of the following Lemma~\ref{l:semigroup_mixed_3rd} is analogous to
the proofs of Lemma~\ref{l:semigroup_partial_diag} and
Lemma~\ref{l:semigroup_mixed_2nd} above and therefore omitted here.
\begin{lemma}[Continuity of mixed third derivatives]\label{l:semigroup_mixed_3rd}
	Assume \textup{Setting~\ref{set:diffusion_ddim}} and
	let $t\in[0,\infty)$ and $f \in C^3([0,1]^d,\R)$. Then it holds for every
	$i,j,l \in \{1,\dotsc,d\}$ that the partial derivative
	\begin{equation}
		[0,1]^d \ni x \mapsto \frac{\partial^3}{\partial x_i\partial x_j\partial
		x_l}\int \bigotimes_{k=1}^d p^k_t(x_k,dy_k) f(y)
	\end{equation}
	exists and is continuous.
\end{lemma}
\begin{lemma}[$C^m$-estimate for multidimensional diffusive part] \label{l:semigroup_diffusion}
	Assume \textup{Setting~\ref{set:diffusion_ddim}},
	let $m\in\{0,1,2,3\}$, let $t\in[0,\infty)$, and let $f \in
	C^m([0,1]^d,\R)$.
	Then it holds that $T^2_t f \in C^m([0,1]^d,\R)$ and
	\begin{equation}
		\norm{T^2_t f}_{C^m} \leq e^{\mu_m t}\norm{ f}_{C^m}.
		\label{eq:semigroup_diffusion_estimate}
	\end{equation}
\end{lemma}
\begin{proof}
	Existence and continuity of the partial derivatives follow from
	Lemma~\ref{l:semigroup_partial_diag}, Lemma~\ref{l:semigroup_mixed_2nd}, and
	Lemma~\ref{l:semigroup_mixed_3rd}.
	It follows from Lemma~\ref{l:semigroup_1d} and from the dominated
	convergence theorem for all $n\in\N_0$ with
	$n\leq m$ and all $x\in [0,1]^d$ that
	\begin{equation}
		\begin{split}
			\abs*{\frac{\partial^n (T^2_t f)}{\partial x_1^n}(x)}
			&= \abs*{\frac{\partial^n}{\partial x_1^n} \int
			p^1_t(x_1,dy_1) \int \bigotimes_{i=2}^d p^i_t(x_i,dy_i) f(y)}\\
			&\leq e^{\mu_n t} \adjustlimits \max_{k\in\{0,\dotsc,n\}} \sup_{z \in [0,1]}\,
			\abs*{\frac{\partial^k}{\partial z^k}\int \bigotimes_{i=2}^d
			p^i_t(x_i,dy_i) f(z,\hat y_1)}\\
			&= e^{\mu_n t} \adjustlimits \max_{k\in\{0,\dotsc,n\}} \sup_{z \in [0,1]}\,
			\abs*{\int \bigotimes_{i=2}^d
			p^i_t(x_i,dy_i)\frac{\partial^k  f}{\partial z^k}(z,\hat y_1)}\\
			&\leq e^{\mu_n t}\max_{k\in\{0,\dotsc,n\}}
			{\norm*{\frac{\partial^k  f}{\partial x_1^k}}_\infty}.
		\end{split}
	\end{equation}
	If $m \geq 2$, Lemma~\ref{l:semigroup_1d} and the dominated convergence
	theorem show for all $x\in [0,1]^d$ that
	\begin{equation}
		\begin{split}
			\abs*{\frac{\partial^2(T^2_t f)}{\partial x_1 \partial
			x_2}(x)}
			&= \abs*{\frac{\partial^2}{\partial x_1 \partial x_2} \int p^1_t(x_1,dy_1)
			\int \bigotimes_{i=2}^d p^i_t(x_i,dy_i) f(y)}\\
			&\leq \adjustlimits \max_{k\in\{0,1\}} \sup_{z_1 \in [0,1]}\,
			\abs*{\frac{\partial^k}{\partial z_1^k}\frac{\partial}{\partial
			x_2}\int \bigotimes_{i=2}^d p^i_t(x_i,dy_i)  f(z_1,\hat y_1)}\\
			&= \adjustlimits \max_{k\in\{0,1\}} \sup_{z_1 \in [0,1]}\,
			\abs*{\frac{\partial}{\partial
			x_2}\int p^2_t(x_2,dy_2) \int \bigotimes_{i=3}^d p^i_t(x_i,dy_i)
			\frac{\partial^k  f}{\partial z_1^k}(z_1,\hat y_1)}\\
			&\leq \adjustlimits \max_{k,l\in\{0,1\}} \sup_{z_1,z_2 \in [0,1]}
			\abs*{\frac{\partial^l}{\partial z_2^l}\int \bigotimes_{k=3}^d p^k_t(x_k,dy_k)
			\frac{\partial^k  f}{\partial z_1^k}(z_1,z_2, \hat y_{12})}\\
			&\leq \max_{k,l\in\{0,1\}}
			{\norm*{
			\frac{\partial^{k+l}  f}{\partial x_1^k\partial x_2^l}}_\infty}.
		\end{split}
	\end{equation}
	Similarly, if $m = 3$, it follows for all $x\in [0,1]^d$ that
	\begin{equation}
		\begin{split}
			\abs*{\frac{\partial^3 (T^2_t f)}{\partial x_1 \partial
			x_2^2}(x)}
			&\leq e^{\mu_2 t}\max_{k\in\{0,1\}, l\in\{0,1,2\}}
			{\norm*{
			\frac{\partial^{k+l} f}{\partial x_1^k\partial x_2^l}}_\infty}
		\end{split}
	\end{equation}
	and
	\begin{equation}
		\begin{split}
			\abs*{\frac{\partial^3 (T^2_t f)}{\partial x_1 \partial x_2 \partial
			x_3}(x)}
			&\leq \max_{k,l,n\in\{0,1\}}
			{\norm*{
			\frac{\partial^{k+l+n} f}{\partial x_1^k\partial x_2^l \partial
			x_3^n}}_\infty}.
		\end{split}
	\end{equation}
	All of the above estimates also hold for the partial derivatives in the
	remaining coordinate directions. Combining all of these estimates
	shows~\eqref{eq:semigroup_diffusion_estimate}. This completes
	the proof of Lemma~\ref{l:semigroup_diffusion}.
\end{proof}

\section{Main result: Spatial derivatives of semigroups} \label{sec:semigroup_estimate}
\begin{theorem}[$C^m$-estimate for semigroups of square-root diffusions] \label{thm:semigroup_together}
	Let $d\in\N$, let $(\Omega,\MCF,\Prob, (\F_t)_{t\in[0,\infty)})$ be a
	stochastic basis, let $W = (W(1),\dotsc,W(d)) \colon [0,\infty) \times \Omega
	\to \R^d$ be a standard $(\F_t)_{t\in[0,\infty)}$-Brownian motion with
	continuous sample paths, let $a_1,\dotsc,a_d \in
	C^3([0,1],\R)$ satisfy for
	all $i\in\{1,\dotsc,d\}$ and all $x\in (0,1)$ that $a_i(0) = 0 = a_i(1)$
	and $a_i(x) > 0$,
	let $b_1,\dotsc,b_d \in C^3([0,1]^d,\R)$ satisfy for all
	$i\in\{1,\dotsc,d\}$ and all $x = (x_1,\dotsc,x_d) \in [0,1]^d$ with
	$x_i \in \{0,1\}$ that $(-1)^{x_i} b_i(x) \geq 0$,
	for every $m\in\{1,2,3\}$ we define 
	$\lambda_m \defeq \max_{\alpha \in
	\N_0^d, 0 < \abs{\alpha} \leq m}\sum_{i=1}^d \norm{\partial^\alpha
	b_i}_\infty$,
	and we define $\lambda_0 \defeq 0$,
	$\mu_0 \defeq 0$, $\mu_1 \defeq 0$,
	$\mu_2 \defeq \max_{i\in\{1,\dotsc,d\}}\frac{1}{2}\norm{\frac{d^2
	a_i}{dx^2}}_\infty$, and
	$\mu_3 \defeq \max_{i\in\{1,\dotsc,d\}}(\norm{\frac{d^3 a_i}{dx^3}}_\infty +
	\frac{3}{2}\norm{\frac{d^2 a_i}{dx^2}}_\infty)$.
	Then
	\begin{enumerate}[label=\textup{(\roman*)}]
		\item there exist
			$(\F_t)_{t\in[0,\infty)}$-adapted processes
			$X^x = (X^x(1),\dotsc,X^x(d)) \colon [0,\infty)\times \Omega \to [0,1]^d$,
			$x\in[0,1]^d$, with
			continuous sample paths satisfying for all $i\in\{1,\dotsc,d\}$, all
			$t\in[0,\infty)$, and all $x = (x_1,\dotsc,x_d) \in[0,1]^d$ that $\Prob$-a.s.
			\begin{equation}\label{eq:SDE.main-result}
				X^x_t(i) = x_i + \int_0^t b_i(X^x_s) \dd s + \int_0^t
				\sqrt{a_i(X^{x}_s(i))}\dd W_s(i)
			\end{equation}
			and
		\item it holds for all $m\in\{0,1,2\}$, all $t\in[0,\infty)$, and all $f \in
			C^{m}([0,1]^d,\R)$ that the function $[0,1]^d \ni x \mapsto \Exp{f(X^x_t)}$ is an
			element of $C^m([0,1]^d,\R)$ and satisfies
			\begin{equation}
				\norm{x \mapsto \Exp{f(X^x_t)}}_{C^m} \leq e^{(m^2 \lambda_m + \mu_m)
				t}\norm{ f}_{C^m}.
			\end{equation}
	\end{enumerate}
\end{theorem}
\begin{proof}
	Theorem~3.2 in~\textcite{ShigaShimizu1980} implies (i).

	We denote by
	$\{T_t\colon t \in[0,\infty)\}$ the family of operators on $C([0,1]^d,\R)$
	that satisfy for all $t\in[0,\infty)$, all $f \in C([0,1]^d,\R)$, and all
	$x\in[0,1]^d$ that $(T_t f)(x) = \Exp{f(X^x_t)}$. Then $\{T_t\colon
	t\in[0,\infty)\}$ is the strongly continuous contraction semigroup on
	$C([0,1]^d,\R)$ associated with the diffusion process $X^{\cdot}$;
	see Remark~3.2 in~\textcite{ShigaShimizu1980}.
	Let $G\colon C^2([0,1]^d,\R) \to C([0,1]^d,\R)$ satisfy for all
	$f \in C^2([0,1]^d,\R)$ and all $x = (x_1,\dotsc,x_d) \in [0,1]^d$ that
	\begin{equation}
		(G f)(x) = \sum_{i=1}^d b_i(x) \frac{\partial  f}{\partial
		x_i}(x) + \frac{1}{2}\sum_{i=1}^d a_i(x_i)
		\frac{\partial^2  f}{\partial
		x_i^2}(x).
	\end{equation}
	Then the generator of $\{T_t \colon t \in[0,\infty)\}$ is given by the closure of
	$G$ (see, e.g., Remark~3.2 in \textcite{ShigaShimizu1980}), so
	$C^2([0,1]^d,\R)$ is a core (cf., e.g., Section I.3 in \textcite{EthierKurtz1986})
	for $G$.
	Let $\{T^1_t \colon t\in[0,\infty)\}$ be as in
	Setting~\ref{set:semigroup_drift}, let $\{T^2_t \colon t\in[0,\infty)\}$ be as in
	Setting~\ref{set:diffusion_ddim}, and let $G_1,
	G_2\colon C^2([0,1]^d,\R) \to C([0,1]^d,\R)$ satisfy for all $ f \in
	C^2([0,1]^d,\R)$ and all $x = (x_1,\dotsc,x_d) \in [0,1]^d$ that
	\begin{equation}
		(G_1 f)(x) = \sum_{i=1}^d b_i(x)\frac{\partial  f}{\partial
		x_i}(x)
	\end{equation}
	and
	\begin{equation}
		(G_2 f)(x) = \frac{1}{2}\sum_{i=1}^d
		a_i(x_i)\frac{\partial^2  f}{\partial x_i^2}(x).
	\end{equation}
	Then the closures of $G_1$ and $G_2$ are the generators of the strongly
	continuous contraction semigroups on $C([0,1]^d,\R)$ given by
	$\{T^1_t\colon t\in[0,\infty)\}$ and $\{T^2_t\colon t\in[0,\infty)\}$,
	respectively. Hence, it holds that $C^2([0,1]^d,\R)$ is a core for $G$,
	that $C^2([0,1]^d,\R)$ is a subset of the
	domains of both $G_1$ and $G_2$, and that $G = G_1 + G_2$ on
	$C^2([0,1]^d,\R)$. Therefore, it follows
	from Trotter's product formula (see, e.g., Corollary I.6.7 in
	\textcite{EthierKurtz1986})
	that the semigroup $\{T_t \colon t\in[0,\infty)\}$ satisfies for all
	$t\in[0,\infty)$ and all $ f\in C([0,1]^d,\R)$ that
	\begin{equation}
		\lim_{n\to\infty} \norm{T_t f -
		(T^1_{\nicefrac tn}T^2_{\nicefrac tn})^n f}_{\infty} = 0.
		\label{eq:trotter_product}
	\end{equation}
	By induction, it follows from Lemma~\ref{l:semigroup_drift_C1},
	Lemma~\ref{l:semigroup_drift_C2}, Lemma~\ref{l:semigroup_drift_C3}, and
	Lemma~\ref{l:semigroup_diffusion} for
	all $n \in \N$, all $m\in\{0,1,2,3\}$, all $t\in[0,\infty)$, and all $f \in
	C^{m}([0,1]^d,\R)$
	that $(T^1_{\nicefrac tn}T^2_{\nicefrac tn})^n f \in C^m([0,1]^d,\R)$ and
	\begin{equation}
		\norm{(T^1_{\nicefrac tn}T^2_{\nicefrac tn})^n  f}_{C^m}
		\leq e^{\left( (m^2 + 4\1_{\{3\}}(m))\lambda_m + \mu_m\right) t}\norm{ f}_{C^m}.
		\label{eq:trotter_product_C2}
	\end{equation}
	Equation~\eqref{eq:trotter_product_C2} shows for all $m\in\{0,1,2\}$,
	all $t\in[0,\infty)$, and
	all $ f \in C^{m+1}([0,1]^d,\R)$ that
	the sequence $\{(T^1_{\nicefrac tn}T^2_{\nicefrac tn})^n f \colon n\in\N\}$ is
	bounded in $C^{m+1}([0,1]^d,\R)$. Therefore, the Arzel\`{a}-Ascoli theorem
	guarantees for all $m\in\{0,1,2\}$, all $t\in[0,\infty)$, and all
	$ f \in C^{m+1}([0,1]^d,\R)$
	that every subsequence of $\{(T^1_{\nicefrac tn}T^2_{\nicefrac tn})^n f
	\colon n\in\N\}$ has
	a convergent subsequence in $C^m([0,1]^d,\R)$,
	whose limit is given by $T_t f$ due to~\eqref{eq:trotter_product}.
	This and~\eqref{eq:trotter_product_C2} imply for all $m\in\{0,1,2\}$, all
	$t\in[0,\infty)$, and all $ f \in C^{m+1}([0,1]^d,\R)$ that
	$T_t f \in C^m([0,1]^d,\R)$ and
	\begin{equation}
		\norm{T_t f}_{C^m} \leq e^{(m^2 \lambda_m + \mu_m) t}\norm{ f}_{C^m}.
		\label{eq:semigroup_together_2}
	\end{equation}

	For the rest of the proof, fix $m\in\{0,1,2\}$, fix $t\in[0,\infty)$, and
	fix $ f \in C^m([0,1]^d,\R)$. Since $C^{m+1}([0,1]^d,\R)$ is dense in
	$C^m([0,1]^d,\R)$, we find a sequence $\{ f_k\colon k\in\N\}
	\subseteq C^{m+1}([0,1]^d,\R)$ with the property that
	$\lim_{k\to\infty}\norm{ f -  f_k}_{C^m} = 0$. By the previous step, it
	holds for all $k\in\N$ that $T_t f_k \in C^m([0,1]^d,\R)$ and for all
	$k,l \in \N$ that
	\begin{equation}
		\norm{T_t f_k - T_t f_l}_{C^m} = \norm{T_t( f_k -
		 f_l)}_{C^m} \leq e^{(m^2 \lambda_m + \mu_m)t}
		\norm{ f_k -  f_l}_{C^m},
	\end{equation}
	which shows that $\{T_t f_k\colon k\in\N\}$ is a Cauchy sequence in
	$C^m([0,1]^d,\R)$. By completeness, it follows that $\{T_t f_k\colon
	k\in\N\}$ converges in $C^m([0,1]^d,\R)$. Moreover, since $T_t$ is a
	contraction on $C([0,1]^d,\R)$, it holds for all $k\in\N$ that
	\begin{equation}
		\norm{T_t f - T_t f_k}_{\infty} = \norm{T_t( f -  f_k)}_{\infty} \leq
		\norm{ f -  f_k}_{\infty}.
	\end{equation}
	This identifies the limit point of $\{T_t f_k \colon k\in\N\} \subseteq
	C^m([0,1]^d,\R)$ and shows that $T_t f \in C^m([0,1]^d,\R)$ and
	that $\lim_{k\to\infty}
	\norm{T_t f - T_t f_k}_{C^m} = 0$. Then it follows
	from~\eqref{eq:semigroup_together_2} that
	\begin{equation}
		\norm{T_t f}_{C^m} = \lim_{k\to\infty}\norm{T_t f_k}_{C^m} \leq
		\lim_{k\to\infty} e^{(m^2 \lambda_m + \mu_m)t}\norm{ f_k}_{C^m} = e^{(m^2 \lambda_m +
		\mu_m)t} \norm{ f}_{C^m}.
	\end{equation}
	Since $m\in\{0,1,2\}$,
	$t\in[0,\infty)$, and $ f \in C^m([0,1]^d,\R)$ were arbitrary, this
	proves (ii) and completes the proof of
	Theorem~\ref{thm:semigroup_together}.
\end{proof}

\subsubsection*{Acknowledgment}
This paper has been partially supported by the DFG Priority Program
``Probabilistic Structures in Evolution'' (SPP 1590), grant HU 1889/3-2.

\printbibliography
\end{document}